\documentclass[11pt]{amsart}
\usepackage{amsmath,amsfonts,amssymb,amsthm,amscd,latexsym,euscript,verbatim}
\usepackage[all]{xy}

\makeatletter
\def\section{\@startsection{section}{1}%
  \z@{1.1\linespacing\@plus\linespacing}{.8\linespacing}%
  {\normalfont\Large\scshape\centering}}
\makeatother

\theoremstyle{plain}

\newtheorem*{conj*}{Root Groups Conjecture}
\newtheorem*{thm1.2}{(1.2) Theorem}
\newtheorem*{thm1.3}{(1.3) Theorem}
\newtheorem*{thm1.4}{(1.4) Theorem}
\newtheorem*{prop*}{Proposition}

\newtheorem{prop}{Proposition}[section]

\newtheorem{thm}[prop]{Theorem}
\newtheorem{cor}[prop]{Corollary}
\newtheorem{lemma}[prop]{Lemma}

\theoremstyle{definition}
\newtheorem{Def}[prop]{Definition}
\newtheorem*{Def*}{Definition}

\newtheorem{notation}[prop]{Notation}
\newtheorem*{notation*}{Notation}
\newtheorem{remark}[prop]{Remark}

\newtheorem*{MT*}{Main Theorem}

\newcommand{\cala}{\mathcal{A}}

\newcommand{\ff}{\mathbb{F}}

\newcommand{\LL}{\mathbb{L}}
\newcommand{\kk}{\mathbb{K}}

\newcommand{\ga}{\alpha}
\newcommand{\gb}{\beta}
\newcommand{\gc}{\gamma}

\newcommand{\gd}{\delta}

\newcommand{\gl}{\lambda}

\newcommand{\gvp}{\varphi}
\newcommand{\gr}{\rho}
\newcommand{\gs}{\sigma}

\newcommand{\half}{\textstyle{\frac{1}{2}}}

\newcommand{\quar}{\textstyle{\frac{1}{4}}}
\newcommand{\tquar}{\textstyle{\frac{3}{4}}}

\numberwithin{equation}{section}

\begin{document}
\title[half-axes in power associative algebras]{half-axes in power associative algebras}
\author[Yoav Segev]{Yoav Segev}
\address{Yoav Segev \\
         Department of Mathematics \\
         Ben-Gurion University \\
         Beer-Sheva 84105 \\
         Israel}
\email{yoavs@math.bgu.ac.il}
%\thanks{$^*$Partially supported by BSF grant no.~2004-083}


\begin{abstract}
Let $A$ be a commutative, non-associative algebra over a field $\ff$
of characteristic $\ne 2$.  A half-axis in $A$ is an idempotent $e\in A$
such that $e$ satisfies the Peirce multiplication rules in a Jordan algebra,
and, in addition, the $1$-eigenspace of ${\rm ad}_e$ (multiplication by $e$)
is one dimensional.

In this paper we consider the identities 

\noindent
$(*)$\quad $x^2x^2=x^4$ and $x^3x^2=xx^4.$

We show that if identities $(*)$ hold strictly in $A,$ then one gets (very)
interesting identities between elements in the  eigenspaces of ${\rm ad}_e$
(note that if $|\ff|>3$ and the identities $(*)$ hold in $A,$ then they
hold strictly in $A$). 
Furthermore we prove that if $A$ is a primitive axial algebra
of Jordan type half (i.e., $A$ is generated by half-axes), and the identities $(*)$ hold
strictly in $A,$ 
then $A$ is a Jordan algebra.  
\end{abstract}

\date{Jan.~22, 2018}
\maketitle
%%%%%%%%%%%%%%%%%%%%%%%%%%%%%%%%%%%%%%%
%%%%%%%%%%%%%%%%%%%%%%%%%%%%%%%%%%%%%%%%
%%%%%%%%%%%%%%%%%%%%%%%%%%%%%%%%%%%
%section1
\section{Introduction}
%%%%%%%%%%%%%%%%%%%%%%%%%%%%%%%%%%%%%%%%%
%%%%%%%%%%%%%%%%%%%%%%%%%%%%%%%%%%%%%
%%%%%%%%%%%%%%%%%%%%%%%%%%%%%%%%%%%%
 %%%%%%%%%%%%%%%%%%%%%%%%%%%%%%%%%%%%

Throughout this paper $\ff$ is a field of characteristic not $2$
and $A$ is a commutative non-associative algebra over $\ff$.
Given  an element  $x\in A$ and a scalar $\gl \in \ff,$ 
we denoted:
\[
A_\gl(x) := \{y \in A\ |\ yx = \lambda y\}\,.
\]
(We allow $A_\gl(x)=0$.)
%%%%%%%%%%%%%%%%%%%%%%%%%%%%%%
%1.1
\begin{Def}\label{def half axis}
%%%%%%%%%%%%%%%%%%%%%%%%%%%%%%
Let $e\in A,$ and set $Z:=A_0(e)$ and $U:=A_{1/2}(e)$.
We say that $e$ is a {\it half-axis} if and only if
\begin{enumerate}
\item
$e^2=e$ (so $e$ is an idempotent).

\item
$A_1(e)=\ff e.$

\item
$A=\ff e\oplus U\oplus Z.$

\item
$Z^2\subseteq Z,\ U^2\subseteq \ff e+Z$ and $UZ\subseteq U.$ 
\end{enumerate}
\end{Def}
Note that any idempotent $e$ in a Jordan algebra $J$ such that $J_1(e)=\ff e$ is a half-axis.

Recall that $A$ is a {\it primitive axial algebra of Jordan type half}
if $A$ is generated (as an algebra over $\ff$) by half-axes.

We also need the following notation.

%%%%%%%%%%%%%%%%%%%%%%%%%%%%%%%%%%%%%%%%%%
%1.2
\begin{notation}
%%%%%%%%%%%%%%%%%%%%%%%%%%%%%%%%%%%%%%%%
Let $e\in A$ be a half-axis, and let $x\in A.$
Write $x=\ga e + x_0+x_{1/2},$ with $\ga\in\ff$ and $x_{\gl}\in A_{\gl}(e),$
for $\gl\in\{0,1/2\}$.   
\begin{enumerate}
\item
We denote $\gvp_e(x)=\gd_x:=\ga.$

\item
We denote $z_x:=x_0$.  We call $z_x$ {\it the $Z$-part of $x$}.
\end{enumerate}
Note that $ex=\gd_x e,$ for $x\in A_1(e)+A_0(e).$ 
\end{notation}
%%%%%%%%%%%%%%%%%%%%%%%%%%%%%%%

Throughout this paper we shall use the technique of
{\it linearization} of identities.  More details about
this technique are given in \S 2.

%%%%%%%%%%%%%%%%%%%%%%%%%%%%%%%%%%%%%%%%%%%%%%%%%%%
%scalar extensions and identities
%1.3
\setcounter{subsection}{2}
\subsection{Scalar extension and strict validity of identities}
%%%%%%%%%%%%%%%%%%%%%%%%%%%%%%%%%%%%%%%%%%%%%%%%%%%%%%%%%%%%%
For a field extension $\kk/\ff$, we denote by $A_\kk := A\otimes_\ff \kk$
the scalar extension (or base change) of $A$ from $\ff$ to $\kk$, which is a commutative non-associative
$\kk$-algebra in a natural way. It is well known that Jordan algebras are invariant under
base change (see e.g.~\cite[Linearization Proposition 1.8.5(2), p.~148]{m}), 
so $A$ is a Jordan algebra over $\ff$ if and only if $A_\kk$ is one over $\kk$. Moreover, since
tensor products commute with direct sums, if $e\in A$ is a half-axis,
then $e$ is a half-axis in $A_\kk$.  Since primitive axial algebras of Jordan type
half are {\it spanned} by half-axes  (see \cite[Corollary 1.2, p.~81]{hrs}),
it follows that primitive axial algebras
are stable under base change as well. 
{\it But power-associative algebras
are not}. For this reason, the concept of strict power-associativity comes in: $A$ is called
{\it strictly power-associative} if the scalar extensions $A_\kk$ are power-associative, for all field
extensions $\kk/\ff$. Similarly, an identity is said to {\it hold strictly} in $A$ if it is satisfied not only
by $A$ but by all its scalar extensions. 
\medskip

Our main result is the following theorem.
%%%%%%%%%%%%%%%%%%%%%%%%%%%%%%%%%%%%%%%%%%%%%
%main theorem
\begin{MT*}\label{thm axial jordan}
%%%%%%%%%%%%%%%%%%%%%%%%%%%%%%%%%%%%%%%%%%%
Let $A$ be primitive axial algebra
of Jordan type half generated over $\ff$ (as an algebra) by a set $\cala$ of half-axes.
The following are equivalent.
\begin{itemize}
\item[$(i)$]
$A$ is a Jordan algebra.
\item[$(ii)$]
$A$ is strictly power associative.
\item[$(iii)$]
The identities
\[
x^2x^2=xx^3\qquad\text{and}\qquad x^3x^2=xx^4
\]
hold strictly in $A.$
\item[$(iv)$]
For all $e\in \cala,\ u\in A_{1/2}(e)$ and $z\in A_0(e)$ we have
\begin{itemize}
\item[$(a)$]
$u^3=\gd_{u^2}u.$

\item[$(b)$]
$(uz)z=\half uz^2.$
\end{itemize}
\end{itemize}
\end{MT*}  
The interesting implications of the Main Theorem are of course $(iii)\Rightarrow (iv)$
and $(iv)\Rightarrow (i)$.  The implication $(iii)\Rightarrow (iv)$ is part of Theorem
\ref{thm identities} below, and the implication $(iv)\Rightarrow (i)$ is proven
in \S 3.
%%%%%%%%%%%%%%%%%%%%%%%%%%%%%%%%%%%%%%%%%%%%%%
%1.3
\begin{remark}\label{rem f>3}
%%%%%%%%%%%%%%%%%%%%%%%%%%%%%%%%%%%%%%%%%%%%
Throughout this paper we will deal with the identities
%eeeeeeeeeeeeeeeeeeeeeeeeeeeeeeeeeeeeeeee
\begin{equation}\label{eq identities}
x^2x^2=xx^3\qquad\text{and}\qquad x^3x^2=xx^4.
\end{equation}
%eeeeeeeeeeeeeeeeeeeeeeeeeeeeeeeeeeee
Note that  if
$|\ff|>3$ and the identities of equation \eqref{eq identities}  hold in $A,$
then, by Corollary \ref{cor O}(2) below, these identities hold strictly in $A$.
Hence if $|\ff|>3$, and  $A$ is a primitive axial algebra of Jordan type half
such that the identities of equation \eqref{eq identities} hold in $A,$ then all the
equivalent properties of the Main Theorem hold in $A.$
\end{remark}

%%%%%%%%%%%%%%%%%%%%%%%%%%%%%%%%%
%1.4
\setcounter{prop}{3}
\begin{thm}\label{thm identities}
%%%%%%%%%%%%%%%%%%%%%%%%%%%%%%%%%
Assume that the identities $x^2x^2=xx^3$ and  $x^3x^2=xx^4$ hold strictly in $A,$
and let $e\in A$ be a half-axis.
Let $u\in A_{1/2}(e)$ and $z\in A_0(e)$.  Then
\begin{enumerate}
%1
\item
$u^3=\gd_{u^2} u.$ %(even when  $|\ff|=3).$

%2
\item
$(uz)z=\half uz^2.$

%3
\item
$4u^2(uz)=2u(u(uz))+u(u^2z)+u^3z.$

%4
\item
$4(uz)^2+2u^2z^2=u(uz^2) +2(u(uz))z+2u((uz)z)+(u^2z)z.$

%5
\item
$4(uz)z^2=uz^3+(uz^2)z+2((uz)z)z.$

%6
\item
$z_{u(uz)}=\half u^2z.$
\end{enumerate}
\end{thm}
 
\noindent
Theorem \ref{thm identities} is proven in \S 2.

%%%%%%%%%%%%%%%%%%%%%%%%%%%%%%%%%%%%%%%%%%%%%%%
%1.5
\begin{remark}
%%%%%%%%%%%%%%%%%%%%%%%%%%%%%%%%%%%%%%%%%%%%
When $|\ff|=3,$ and the identities of equation \eqref{eq identities}
hold in $A,$ one can show that identities (1), (4) and (6) of Theorem \ref{thm identities}
remain valid.
It is quite possible that additional (long) calculations will show
that the other identities are valid as well.  Once this is done
it will follow from the implication $(iv)\Rightarrow (i)$ of the Main
Theorem that if $A$ is a primitive axial algebra of Jordan type half
and the identities of equation \eqref{eq identities}  hold in $A,$ then 
$A$ is a Jordan algebra. (Indeed, by Remark \ref{rem f>3}, if $|\ff|>3$
and the identities of equation \eqref{eq identities} hold in $A,$ then they
hold strictly in $A,$ so Theorem \ref{thm identities} applies, and then the
Main Theorem applies.)
\end{remark}

We mention that Theorem \ref{thm jordan} in \S 7  gives interesting
necessary and sufficient conditions for a commutative non-associative algebra $A$
having a half-axis to be a Jordan algebra. See also Theorem \ref{thm jordan 2}.
Furthermore, throughout
this paper we obtain additional various interesting identities. 
 
%%%%%%%%%%%%%%%%%%%%%%%%%%%%%%%%%%%%%%%%%%
%1.6
\begin{notation}\label{not main}
%%%%%%%%%%%%%%%%%%%%%%%%%%%%%%%%%%%%%%%%%%
Throughout this paper we let  $e$ be a fixed half-axis in $A$.
We denote
\[
U:=A_{1/2}(e)\quad\text{and}\quad Z=A_0(e).
\]
\end{notation}

To conclude the introduction we would like to mention that
this paper was inspired by \cite{tb}.

%%%%%%%%%%%%%%%%%%%%%%%%%%%%%%%%%%%%%%%%%%%%%%%%
%%%%%%%%%%%%%%%%%%%%%%%%%%%%%%%%%%%%%%%%%%%%%%%%%
%%%%%%%%%%%%%%%%%%%%%%%%%%%%%%%%%%%%%%%%%%%%%%%
%section2
\section{The proof of Theorem \ref{thm identities}}
%%%%%%%%%%%%%%%%%%%%%%%%%%%%%%%%%%%%%%%%%%%%%%%
%%%%%%%%%%%%%%%%%%%%%%%%%%%%%%%%%%%%%%%%%%%%%%
%%%%%%%%%%%%%%%%%%%%%%%%%%%%%%%%%%%%%%%%%%%%

In this section we adopt the notation and terminology of \cite{o}.
Thus we consider non-associative and non-commutative polynomials $f(x_1,\dots, x_m)$
over $\ff$.  The degree of $x_i$ in each monomial of $f$ is defined
on p.~167 of \cite{o} and $f$ is {\it homogeneous} 
if, for each $i\in\{1,\dots,m\}$, the degree of $x_i$ is the same in all of the
monomials of $f$.  

Recall also from p.~176 of \cite{o} the notion of {\it linearization} of $f$ 
and the notion of {\it  a stable derivative} of $f$. Note that {\it all stable derivatives
of $f$ contain all linearizations of $f$}.  A {\it homogeneous identity}
is an identity $f(x_1,\dots,x_m)=0,$ where $f$ is a homogeneous polynomial.
We will not distinguish between the polynomial $f$ and the identity $f=0.$

Corollary \ref{cor O} below is well known.  However, since
we did not find an explicit reference to it, we include a proof.
First we quote the following two results from \cite{o}.

%%%%%%%%%%%%%%%%%%%%%%%%%%%%%%%%%%%%%%%%
%2.1
\begin{thm}[Theorem 3.5 in \cite{o}]\label{thm O}
%%%%%%%%%%%%%%%%%%%%%%%%%%%%%%%%%%%%%%%
Let $\LL$ be a field with $|\LL|\ge d$, let $f(x_1,\dots,x_m)$ be a
homogeneous identity over $\LL,$ such that the degree of each $x_i, i=1,\dots,m,$ in $f$ is no
more than $d$.  Let $A$ be a $\LL$-algebra satisfying
$f$. Then $A$ satisfies all stable derivatives of $f$.
\end{thm}

%%%%%%%%%%%%%%%%%%%%%%%%%%%%%%%%%%%%%%%%%%%%
%2.2
\begin{prop}\label{prop O}
%%%%%%%%%%%%%%%%%%%%%%%%%%%%%%%%%%%%%%%%%
Let $S$ be a set of homogeneous identities over a field $\kk,$
and let $V$ be the variety of
$\kk$-algebras determined by $S$. Let $A \in V$ satisfy every stable derivative
of each identity in $S$. Then, $A_\LL\in V_\LL$, for every field extension 
$\LL/\kk,$ where $V_{\LL}$ is the variety of $\LL$-algebras determined by $S.$
\end{prop}
\begin{proof}
See \cite[Proposition 4.2]{o}.
\end{proof}
 
%%%%%%%%%%%%%%%%%%%%%%%%%%%%%%%%%%%%%%%%%%%%%
%2.3
\begin{cor}\label{cor O}
%%%%%%%%%%%%%%%%%%%%%%%%%%%%%%%%%%%%%%%%%%%%
Let $f(x_1,\dots,x_m)=0$ be a homogeneous identity over a field $\kk$.  Then
\begin{enumerate}
\item
$f$ holds strictly in $A$ if and only if  $A$ satisfies all stable derivatives of $f.$

\item
If the degree of each $x_i,\ i=1,\dots m,$ in $f$ is at most $d,$ and $|\ff|\ge d,$
then  $f$ holds strictly in $A.$
\end{enumerate}
\end{cor}
\begin{proof}
(1):\quad Suppose $f$ holds strictly in $A$.  Let $\LL/\kk$ be an infinite field extension.
Then $\ff$ holds in $A_\LL,$ and $\LL$ satisfies the hypothesis of Therem \ref{thm O}.
Hence $A_\LL$ satisfies every stable derivative of $f.$  Since $A$ embeds in $A_{\LL},$
$A$ satisfies every stable derivative of $f.$

Conversely, suppose  $A$ satisfies all stable derivatives of $f$.
Then by Proposition \ref{prop O}, $f$ holds strictly in $\kk.$
\medskip

\noindent
(2):\quad 
Assume the hypotheses of (2).  By Theorem \ref{thm O},
$A$ satisfies every stable derivative of $f$.  Hence (2) follows from (1).
\end{proof}

In this section we prove Theorem \ref{thm identities}.
The proof uses linearization techniques.
We repeatedly use Corollary \ref{cor O} without further reference.

%%%%%%%%%%%%%%%%%%%%%%%%%%%%%%%%%%%%%%%%%%%%%%
%2.4
\begin{lemma}\label{lem thm 1.3}
Assume that the identity $x^2x^2=x^3x$ holds strictly in $A$ $($in particular
this holds if $|\ff|>3)$.  Then the following identities hold in $A$
%eeeeeeeeeeeeeeeeeeeeeeeeeeeeeeeee
\begin{equation}\label{eq x^2x^2 1}
4x^2(xy)=x^3y+x(x^2y)+2x(x(xy)).
\end{equation}
%eeeeeeeeeeeeeeeeeeeeeeeeeeeeeeeeeeee
%
%eeeeeeeeeeeeeeeeeeeeeeeeeeeeeeeeeeeeeeeeeeeeeeeeeeee
\begin{gather}\label{eq x^2x^2 2}
8(xy)(xw)+4x^2(yw)=(x^2y)w+(x^2w)y+2(x(xy))w+2(x(xw))y\\\notag
+2x((xy)w)+2x((xw)y)+2x((x(yw)).
\end{gather}
%eeeeeeeeeeeeeeeeeeeeeeeeeeeeeeeeeeeeeeeeeeeeeeeeeeee
\end{lemma}
\begin{proof}
Linearizing the identity $x^2x^2=x^3x$ at $x$ in the direction $y$ we get
\[
(x^2+2xy+y^2)(x^2+2xy+y^2)=[(x^2+2xy+y^2)(x+y)](x+y).
\]
Equating only the expressions in which $y$ is linear we get
$4x^2(xy)=x^3y+x(x^2y)+2x(x(xy)),$ which is equation \eqref{eq x^2x^2 1}.

Next, linearizing equation \eqref{eq x^2x^2 1} at $x$ in the direction $w$ we get
\begin{align*}
&4(x+w)^2((x+ w)y)\\
&=(x+w)^3y+(x+ w)((x+ w)^2y)+2(x+ w)((x+ w)((x+ w)y)),
\end{align*}
or
\begin{align*}
&4(x^2+2xw+w^2)(xy+wy)\\
&=\big((x^2+2xw+w^2)(x+w)\big)y+(x+w)\big((x^2+2xw+w^2)y\big)\\
&+2(x+ w)\big((x+ w)(xy+ wy)\big).
\end{align*}
Equating only the expressions in which $w$ is linear we get
\begin{gather*}
4x^2(yw)+8(xy)(xw)=[(x^2w)y+2(x(xw))y]+[2x((xw)y)+(x^2y)w]\\
+[2x(x(wy))+2x((xy)w)+2(x(xy))w].
\end{gather*}
rearranging we get equation \eqref{eq x^2x^2 2}.
\end{proof}

%%%%%%%%%%%%%%%%%%%%%%%%%%%%%%%%%%%%%%%%
%2.5
\begin{cor}\label{cor thm 1.4}
%%%%%%%%%%%%%%%%%%%%%%%%%%%%%%%%%%%%%%%
Assume that the identity $x^2x^2=x^3x$ holds strictly in $A$ 
$($in particular this holds if $|\ff|>3)$.  Then
\begin{enumerate}
\item
identities $(3), (4)$ and $(5)$ of Theorem \ref{thm identities}
hold with $u$ replaced by any $x\in A$ and $z$ replaced by any $y\in A.$
\item
Identities $(1)$ and $(6)$ of Theorem \ref{thm identities}
hold in $A.$
\end{enumerate}
\end{cor}
\begin{proof}
(1):\quad Replacing $u$ by $x$ and $z$ by $y$ in identity (3) of Theorem \ref{thm identities},
we get identity \eqref{eq x^2x^2 1}.  Similarly
we get identity (5) of Theorem \ref{thm identities}.  Putting $w=y$ in identity
\eqref{eq x^2x^2 2} we get the more general version of identity (4) of Theorem \ref{thm identities}.

(2):\quad Letting $y=e$ and $x=u\in U$
in identity \eqref{eq x^2x^2 1}, we get
$2u^3=\half u^3+\half\gd_{u^2}u+u^3,$ and identity (1) of Theorem \ref{thm identities}
follows.  

Finally let $w=e,\, x=u\in U$ and $y=z\in Z$ in identity \eqref{eq x^2x^2 2}.  Then $(x^2y)e=(x^2e)y=0$ and hence
by identity \eqref{eq x^2x^2 2} we get
\[
4u(uz)=2e(u(uz))+u^2z+u(uz)+u(uz).
\]
It follows that $u^2z=2(u(uz)-e(u(uz))),$ this shows identity (6) of Theorem \ref{thm identities}.
\end{proof}

%%%%%%%%%%%%%%%%%%%%%%%%%%%%%%%%%%%%%%%%
%2.6
\begin{prop}\label{prop uzz}
%%%%%%%%%%%%%%%%%%%%%%%%%%%%%%%%%%%%%%%%
Assume that the identity $(xx^2)x^2=x(x^2x^2)$ holds strictly in $A$ 
$($in particular this holds if $|\ff|>3)$.  Then idenity (2) of 
Theorem \ref{thm identities} holds in $A.$
\end{prop}
\begin{proof}
We write the identity $(xx^2)x^2=x(x^2x^2)$ as
%eeeeeeeeeeeeeeeeeeeeeeeeeeeeeeeeeeeeeeeeeeee
\begin{equation}\label{eq x^3x^2=xx^4}
[x,x^2,x^2]=0,
\end{equation}
%eeeeeeeeeeeeeeeeeeeeeeeeeeeeeeeeeeeeeeeeeeee
in terms of the (trilinear) associator $[x,y,z]=(xy)z-x(yz)$.
Liniarizing identity \eqref{eq x^3x^2=xx^4} in the direction of $y$ we get
%eeeeeeeeeeeeeeeeeeeeeeeeeeeeeeeeee
\begin{equation}\label{eq lin y}
[y,x^2,x^2]+2[x,xy,x^2]+2[x,x^2,xy]=0.
\end{equation}
%eeeeeeeeeeeeeeeeeeeeeeeeeeeeeeeeeeeeeeeeee
Next we claim that the identity
%eeeeeeeeeeeeeeeeeeeeeeeeeeeeeeeeeeeeeeeeeeeee
\begin{align}\label{eq lin v}
0&=[x,yv,x^2]+[y,vx,x^2]+[v,xy,x^2]+[x,x^2,yv]\\\notag
&+[y,x^2,vx]+[v,x^2,xy]+2[x,xy,xv]+2[x,xv,xy],
\end{align}
holds in $A$.  Indeed, linearizing identity \eqref{eq lin y} at $x$ in the direction $v$ we get
\begin{gather*}
[y,(x+v)^2,(x+v)^2]+2[(x+v),(x+v)y,(x+v)^2]\\
+2[(x+v),(x+v)^2,(x+v)y]=0.
\end{gather*}
So
\begin{gather*}
\Big(2[y,xv,x^2]+2[y,x^2,xv]\Big)+2\Big([x,vy,x^2]+2[x,xy,xv]+[v,xy,x^2]\Big)\\
+2\Big([x,x^2,vy]+2[x,xv,xy]+[v,x^2,xy]\Big)=0,
\end{gather*}
dividing by $2$ and rearranging we get identity \eqref{eq lin v}.

Next linearizing identity \eqref{eq lin v} at $x$ in the direction $w$
we get
\begin{align*}
0&=[x+w,yv,(x+ w)^2]+[y,v(x+ w),(x+ w)^2]\\
&+[v,(x+ w)y,(x+ w)^2]+[x+ w,(x+ w)^2,yv]\\
&+[y,(x+ w)^2,v(x+ w)]+[v,(x+ w)^2,(x+ w)y]\\
&+2[x+ w,(x+ w)y,(x+ w)v]+2[x+ w,(x+ w)v,(x+ w)y].
\end{align*}
Doing similar calculations and rearranging we get
\begin{align}\label{eq lin w}
0&= [y, vw, x^2] + [v,wy, x^2] + [w, yv, x^2]\\\notag
&+ 2[y, xv, xw] + 2[v, xw, xy] + 2[w, xy, xv]\\\notag
&+ 2[y, xw, xv] + 2[w, xv, xy] + 2[v, xy, xw]\\\notag
&+ [y, x^2, vw] + [v, x^2,wy] + [w, x^2, yv]\\\notag
&+ 2[x, yv, xw] + 2[x, vw, xy] + 2[x,wy, xv]\\\notag
&+ 2[x, xy, vw] + 2[x, xv,wy] + 2[x, xw, yv].
\end{align}
Here we put $w = v = e$ and $x=z\in Z, y=u\in U,$ to deduce
\begin{align*}
0&= [u, e, z^2] + [e,eu, z^2] + [e, ue, z^2]\\\notag
&+ 2[u, ze, ze] + 2[e, ze, zu] + 2[e, zu, ze]\\\notag
&+ 2[u, ze, ze] + 2[e, ze, zu] + 2[e, zu, ze]\\\notag
&+ [u, z^2, ee] + [e, z^2,eu] + [e, z^2, ue]\\\notag
&+ 2[z, ue, ze] + 2[z, ee, zu] + 2[z,eu, ze]\\\notag
&+ 2[z, zu, ee] + 2[z, ze,eu] + 2[z, ze, ue].
\end{align*}
So
 \begin{align*}
0&= [u, e, z^2] + 2[e,eu, z^2]\\\notag
&+ 0 + 0 + 0\\\notag
&+ 0 + 0 + 0\\\notag
&+ [u, z^2, e] + [e, z^2,eu] + [e, z^2, ue]\\\notag
&+ 0 + 2[z, e, zu] + 0\\\notag
&+ 2[z, zu, e] + 0 + 0.
\end{align*}
The fusion rules now yield: 
\[
0 = \half uz^2 + 0 +\half uz^2 - \quar uz^2-\quar uz^2 -z(zu)+0.
\]
This shows that identity (2) of Theorem \ref{thm identities} holds.
\end{proof}

Of course Theorem \ref{thm identities} follows from Corollary \ref{cor thm 1.4}
and Proposition \ref{prop uzz}.
 
%%%%%%%%%%%%%%%%%%%%%%%%%%%%%%%%%%%%%%%%%%%%%%%%%%%%%%%%%%%
%%%%%%%%%%%%%%%%%%%%%%%%%%%%%%%%%%%%%%%%%%%%%%%%%%%%%%%%%
%%%%%%%%%%%%%%%%%%%%%%%%%%%%%%%%%%%%%%%%%%%%%%%%%%%%%
%section3
\section{The proof of the Main Theorem}
%%%%%%%%%%%%%%%%%%%%%%%%%%%%%%%%%%%%%%%%%%%%%%%%%%%
%%%%%%%%%%%%%%%%%%%%%%%%%%%%%%%%%%%%%%%%%%%%%%%%%
%%%%%%%%%%%%%%%%%%%%%%%%%%%%%%%%%%%%%%%%%%%%%%
In this section $A$ is a commutative non-associative algebra
(at the moment we do not put any additional hypotheses on $A$).
We now prove the Main Theorem.

The implication $(i)\Rightarrow (ii)$ is well known and follows from
the fact that Jordan algebras are stable under base change and are power associative. The implication
$(ii)\Rightarrow (iii)$ is obvious, while the implication $(iii)\Rightarrow (iv)$
follows from Theorem \ref{thm identities}.

We now prove the implication $(iv)\Rightarrow (i)$.
We start with:
 
%%%%%%%%%%%%%%%%%%%%%%%%%%%%%%%%%%%%%%%%%%
%3.1
\begin{lemma}\label{lem u z1}
%%%%%%%%%%%%%%%%%%%%%%%%%%%%%%%%%%%%%%%%%
Let $u\in U$.  Then the following are equivalent:
\begin{itemize}
\item[$(i)$]
$u^3=\gd_{u^2}u.$

\item[$(ii)$]
$uz_{u^2}=\half\gd_{u^2} u.$

\item[$(iii)$]
$(eu^2)u=\half u^3.$
\end{itemize}
Furthermore, if (i)--(iii) hold and $u^2u^2=u^3u,$
then $z_{u^2}^2=\gd_{u^2} z_{u^2}.$
\end{lemma}
\begin{proof}
$(i)\iff(ii):$\quad
We have
\[
u^3=u^2u=(\gd_{u^2} e+z_{u^2})u=\half\gd_{u^2} u+uz_{u^2}.
\]
Hence $(i)$ and $(ii)$ are equivalent.
\medskip

\noindent
$(i)\iff(iii):$\quad
Since $(eu^2)u=(\gd_{u^2}e)u=\half \gd_{u^2}u$ we see that $(i)$ and $(iii)$ are equivalent.

Suppose that $u^2u^2=u^3u$ and that $(i)$ holds.  Then
\[
\gd_{u^2}^2e+z_{u^2}^2=(u^2)^2=u^3u=\gd_{u^2} u^2=\gd_{u^2}^2 e+\gd_{u^2} z_{u^2}.\qedhere
\]
\end{proof}

%%%%%%%%%%%%%%%%%%%%%%%%%%%%%%%%%%%%%%%%%%%
%3.2
\begin{lemma}\label{lem jordan e}
%%%%%%%%%%%%%%%%%%%%%%%%%%%%%%%%%%%%%%%%
Assume that $u^3=\gd_{u^2}u$ and that $(uz)z=\half uz^2,$ for all
$u\in U$ and $z\in Z$.  Then
$x(x^2e)=x^2(xe),$ for all $x\in A$.  
\end{lemma}
\begin{proof} 
Write $x=\ga e +u+z,$ with $u\in U$ and $z\in Z$.  Then
\begin{align*}
x^2 &=\ga^2e + u^2+ z^2+\ga u + 2uz.\\
xe &=\ga e+\half u.\\
x^2e &=\ga^2e+eu^2+\half\ga u+uz.
\end{align*}
Hence
\begin{align*}
&(x^2e)x =(\ga^2e+eu^2+\half\ga u+uz)(\ga e +u+z)\\
&=\ga^3 e+\tquar\ga^2u+\ga eu^2+\half u^3+\half \ga u^2+\ga uz+u(uz)+(uz)z.
\end{align*}
Indeed $\tquar\ga^2u$ is obtained from $(\ga^2e) u$ and $(\half\ga u)(\ga e)$.
And $\ga uz$ is obtained from $(\half \ga u) z$ and $(uz) (\ga e)$.
Further, the equality $(eu^2) u=\half u^3$  holds by hypothesis, and by Lemma \ref{lem u z1}.

Next we compute
\begin{align*}
x^2(xe) &=(\ga^2e + u^2+ z^2+\ga u + 2uz)(\ga e+\half u)\\
&=\ga^3 e+\ga e u^2+\tquar\ga^2 u+\ga uz\\
&+\half u^3+\half uz^2+\half\ga u^2+u(uz).
\end{align*}
Indeed $\tquar\ga^2u$ is obtained from $(\ga u) (\ga e)$ and $(\ga^2 e) (\half u)$.
Note now that all that remains to show in order to show that $(x^2e)x=x^2(xe),$ is the equality 
$(uz)z=\half uz^2$.  This equality holds by hypothesis.
\end{proof}
 
We can now prove the implication $(iv)\Rightarrow (i)$ of the Main Theorem.

\noindent
Suppose that $A$ is a primitive axial algebra of Jordan type half generated by a set of half-axes $\cala$.
By \cite[Corollary 1.2, p.~81]{hrs} (see also \cite[Corollary 3.4]{hss2})
we may assume that $A$ is spanned by $\cala$.  Hence to show the identity
$x^2(xy)=x(x^2y),$ for all $x,y\in A,$ it suffices to show this identity
when $y=f$ is an arbitrary half-axis $f\in\cala$.  Since $e$ is an arbitrary half-axis
in $A,$ Lemma \ref{lem jordan e},
and hypotheses  $(iv)$ of the Main Theorem show that  $x^2(xf)=x(x^2f),$ 
for all half-axes $f\in\cala$.

%%%%%%%%%%%%%%%%%%%%%%%%%%%%%%%%%%%%%%%%%%%%%%%%%%
%%%%%%%%%%%%%%%%%%%%%%%%%%%%%%%%%%%%%%%%%%%%%%%%%%%
%%%%%%%%%%%%%%%%%%%%%%%%%%%%%%%%%%%%%%%%%%%%%%%%%
%section4
\section{Consequences of Theorem \ref{thm identities}}
%%%%%%%%%%%%%%%%%%%%%%%%%%%%%%%%%%%%%%%%%%%%%%%%%%
%%%%%%%%%%%%%%%%%%%%%%%%%%%%%%%%%%%%%%%%%%%%%%%%
%%%%%%%%%%%%%%%%%%%%%%%%%%%%%%%%%%%%%%%%%%%%%%%%%
In this section we prove some consequences of Theorem
\ref{thm identities} which are useful to know and
which will be applied in the following sections.
Thus, throughout this section
we assume that the identities $x^2x^2=xx^3$ and  $x^3x^2=xx^4$ hold strictly in $A.$
%%%%%%%%%%%%%%%%%%%%%%%%%%%%%%%%%%%%%%%
%4.1
\begin{lemma}\label{lem (u_1u_2)z}
%%%%%%%%%%%%%%%%%%%%%%%%%%%%%%%%%%%%%%
Let $u_1, u_2\in U,$ and $z\in Z$.  Then 
\[
u_1(u_2z)+u_2(u_1z)=\gr e+(u_1u_2)z,
\]
where $\gr=\gvp_e(u_1(u_2z))+\gvp_e(u_2(u_1z)).$
\end{lemma}
This follows immediately from linearization of Theorem \ref{thm identities}(6).
\begin{proof}
For the convenience of the reader we give the details.
Using Theorem \ref{thm identities}(6), write $u_i(u_iz)=\gc_i e+\half u_i^2z,$ with $\gc_i\in\ff,$ for $i\in \{1, 2\}$.  Then
by Theorem \ref{thm identities}(6),
\begin{align*}
&(u_1+u_2)((u_1+u_2)z)=\gc e+\half(u_1+u_2)^2z, \text{ for some $\gc\in\ff$}\\
&\iff u_1(u_1z)+u_2(u_2z)+u_1(u_2z)+u_2(u_1z)\\
&=\gc e+\half u_1^2z+\half u_2^2z+(u_1u_2)z\\
&\iff \gc_1e+\half u_1^2z+\gc_2 e+\half u_2^2z+u_1(u_2z)+u_2(u_1z)\\
&=\gc e+\half u_1^2z+\half u_2^2z+(u_1u_2)z\\
&\iff (\gc_1+\gc_2)e+u_1(u_2z)+u_2(u_1z)=\gc e+(u_1u_2)z\\
&\iff  u_1(u_2z)+u_2(u_1z)=\gr e+(u_1u_2)z.
\end{align*}
The lemma follows.
\end{proof}

%%%%%%%%%%%%%%%%%%%%%%%%%%%%%%%%%%%%%
%4.2
\begin{lemma}\label{lem u(z_1z_2)}
%%%%%%%%%%%%%%%%%%%%%%%%%%%%%%%%%%%%%
Let $u\in U$ and $z_1, z_2, z\in Z$. Then
\begin{enumerate}
\item
$u(z_1z_2)=(uz_1)z_2+(uz_2)z_1.$

\item
$(uz^2)z=(uz)z^2=\half uz^3.$
\end{enumerate}
\end{lemma}
\begin{proof}
(1):\quad
By Theorem \ref{thm identities}(2) we have
\begin{align*}
&2(u(z_1+z_2))(z_1+z_2)=u(z_1+z_2)^2\quad\iff\\
&2(uz_1)z_1+2(uz_2)z_2+2[(uz_1)z_2+(uz_2)z_1]=uz_1^2+uz_2^2+2u(z_1z_2).
\end{align*}
Since $2(uz_i)z_i=uz_i^2,$ for $i=1,2,$ part (1) holds.
\medskip

\noindent
(2):   Put  $x = z,$ $y = e$ and $v=u$ in identity \eqref{eq lin v} to conclude
\begin{align*}
0&=[z,eu,z^2]+[e,uz,z^2]+[u,ez,z^2]+[z,z^2,eu]\\\notag
&+[e,z^2,uz]+[u,z^2,ze]+2[x,ze,zu]+2[z,zu,ze],
\end{align*}
That is
\[
0=\half (uz)z^2-\half (uz^2)z+0+0+\half uz^3-\half (uz^2)z-\half (uz)z^2+0+0+0.
\]
Hence $(uz^2)z=\half uz^3$.  Also by (1), $uz^3=(uz^2)z+(uz)z^2,$
so (2) holds.
\end{proof}

%%%%%%%%%%%%%%%%%%%%%%%%%%%%%%%%%%%%%%
%4.3
\begin{lemma}\label{lem u_1(u_1u_2)}
%%%%%%%%%%%%%%%%%%%%%%%%%%%%%%%%%%%%%%%
Let $u_1, u_2\in U$ and write $u_1^2=\gd_1e+z_1$ and $u_2^2=\gd_2e+z_2,$
with $\gd_i\in \ff$ and $z_i\in Z,$ for $i=1,2.$.

Let
\begin{align*}
u_1u_2&=\gd_{12}e+z_{12}\text{ and }(u_1+u_2)^2=\gd_{1+2}e+z_{1+2},\\
&\gd_{12},\gd_{1+2}\in\ff, z_{12},z_{1+2}\in Z.
\end{align*}
Then
\begin{enumerate}
%1
\item
$\gd_{1+2}=\gd_1+\gd_2+2\gd_{12}.$

%2
\item
\begin{gather*}
\gd_{1+2}(u_1+u_2)=\\
\gd_1u_1+\gd_2u_2+u_1^2u_2+u_2^2u_1+\gd_{12}u_1+\gd_{12}u_2+2z_{12}u_1+2z_{12}u_2.
\end{gather*}
 
%3
\item
$u_1^2u_2+u_2^2u_1-\gd_{12}u_1-\gd_2u_1-\gd_{12}u_2-\gd_1u_2+2z_{12}(u_1+u_2)=0.$ 

%4
\item
$u_1^2u_2-\gd_{12}u_1-\gd_1u_2+2u_1z_{12}=0.$

%5
\item
$2u_1(u_1u_2)=-u_1^2u_2+\gd_1u_2+2\gd_{12}u_1.$

%6
\item
$u_2z_1-\gd_{12}u_1-\half\gd_1u_2+2u_1z_{12} =0.$

%7
\item
$2(u_1z_{12})z_1=\half\gd_{12}\gd_1 u_1.$

%8
\item
$2(u_1(u_1u_2))z_1=\gd_{12}\gd_1 u_1.$
\end{enumerate}
\end{lemma}
\begin{proof}
(1):\quad We have 
\begin{align*}
\gd_{1+2}e+z_{1+2}&=(u_1+u_2)^2=u_1^2+u_2^2+2u_1u_2\\
&=(\gd_1+\gd_2+2\gd_{12})e+z_1+z_2+2z_{12}.
\end{align*}
\medskip

\noindent
(2):\quad We have
\begin{align*}
\gd_{1+2}(u_1+u_2)&=(u_1+u_2)^3=u_1^3+u_2^3+u_1^2u_2+u_2^2u_1+2(u_1u_2)u_1+2(u_1u_2)u_2\\
&=\gd_1u_1+\gd_2u_2+u_1^2u_2+u_2^2u_1+\gd_{12}u_1+2z_{12}u_1+\gd_{12}u_2+2z_{12}u_2.
\end{align*}
\medskip

\noindent
(3):\quad Just replace $\gd_{1+2}$ by $\gd_1+\gd_2+2\gd_{12}$ in (2).
\medskip

\noindent
(4):\quad Replace $u_1$ with $\ga u_1$ in (3), where $0\ne\ga\in\ff$.  Then $\gd_1$ should be replaced with $\ga^2\gd_1,$
$\gd_{12}$ with $\ga\gd_{12}$ and $z_{12}$ with $\ga z_{12}$.  We get from (3),
\[
\ga^2u_1^2u_2+\ga u_2^2u_1-\ga^2\gd_{12}u_1-\ga\gd_2u_1-\ga\gd_{12}u_2-\ga^2\gd_1u_2+2\ga z_{12}(\ga u_1+u_2)=0.
\]
dividing by $\ga$ we see that
\[\tag{$*$}
\ga u_1^2u_2+u_2^2u_1-\ga\gd_{12}u_1-\gd_2u_1-\gd_{12}u_2-\ga\gd_1u_2+2 z_{12}(\ga u_1+u_2)=0.
\]
Since $\ga$ is arbitrary we get (4). 
\medskip

\noindent
(5):\quad
Replacing in (4) $z_{12}=u_1u_2-\gd_{12}e$ we get
\begin{align*}
&u_1^2u_2-\gd_{12}u_1-\gd_1u_2+2(u_1u_2-\gd_{12}e)u_1=0\\
&\iff 2u_1(u_1u_2)=-u_1^2u_2+\gd_1u_2+2\gd_{12}u_1.
\end{align*}
\medskip

\noindent
(6):  This is obtained from (4) by replacing $u_1^2$ with $\gd_1 e+z_1.$
\medskip

\noindent
(7):\quad
Now multiply (6) by $z_1$.  We get
\begin{align*}
&(u_2z_1)z_1-\gd_{12}u_1z_1-\half\gd_1u_2z_1+2(u_1z_{12})z_1=0\\
&\overset{(i)}{\iff} \half u_2z_1^2-\half\gd_1\gd_{12}u_1-\half\gd_1u_2z_1+2(u_1z_{12})z_1=0\\
&\overset{(ii)}{\iff} \half \gd_1 u_2z_1-\half\gd_1\gd_{12}u_1-\half\gd_1 u_2z_1+2(u_1z_{12})z_1=0\\
&\iff -\half\gd_1\gd_{12}u_1+2(u_1z_{12})z_1=0.
\end{align*}
Here (i) holds because by Lemma \ref{lem u z1}, $u_1z_1=\half\gd_1u_1,$
and (ii) holds by Lemma \ref{lem u z1}.
\medskip

\noindent
(8):\quad
Replacing   $z_{12}$ by $u_1u_2-\gd_{12}e$ in (7), we get
\begin{align*}
&2(u_1z_{12})z_1=\half\gd_{12}\gd_1 u_1\\
&\iff 2(u_1(u_1u_2-\gd_{12}e))z_1=\half\gd_{12}\gd_1 u_1\\
&\iff 2(u_1(u_1u_2))z_1-\gd_{12}u_1z_1=\half\gd_{12}\gd_1 u_1\\
&\overset{(i)}{\iff} 2(u_1(u_1u_2))z_1-\half  \gd_{12}\gd_1u_1=\half\gd_{12}\gd_1 u_1\\
&\iff 2(u_1(u_1u_2))z_1=\gd_{12}\gd_1 u_1.
\end{align*}
Where (i) holds by Lemma \ref{lem u z1}.
\end{proof}

%%%%%%%%%%%%%%%%%%%%%%%%%%%%%%%%%%%%%%%%%
%4.4
\begin{lemma}\label{lem u_1u_2u_3}
%%%%%%%%%%%%%%%%%%%%%%%%%%%%%%%%%%%%%%%
Let $u_1, u_2, u_3\in U$.  Then
\begin{enumerate}
%1
\item
$u_1(u_2u_3)+(u_1u_2)u_3+(u_1u_3)u_2=\gd_{u_2u_3}u_1+\gd_{u_1u_3}u_2+\gd_{u_1u_2}u_3.$

%2
\item
$u_1^2u_2+2u_1(u_1u_2)=\gd_{u_1^2}u_2+2\gd_{u_1u_2}u_1.$

%3
\item
$u_1^2(u_2z)+2u_1(u_1(u_2z))=\gd_{u_1^2}u_2z+2\gd_{u_1(u_2z)}u_1.$

%4
\item
$(u_1^2u_2)z=-2u_1((u_1u_2)z)+2(u_1u_2)(u_1z) +\gd_{u_1^2}u_2z.$

%5
\item
$2u_1^2(u_1z)=\gd_{u_1^2}u_1z+\gd_{u_1(u_1z)}u_1.$
\end{enumerate}
\end{lemma}
\begin{proof}
(1):\quad
Put $u_2+u_3$ in place of $u_1,$  and $u_1$ in place of $u_2$ in Lemma \ref{lem u_1(u_1u_2)}(5) to get:
\begin{align*}
&2(u_1(u_2+u_3))(u_2+u_3)=\\
&-u_1(u_2+u_3)^2+\gd_{(u_2+u_3)^2}u_1+2\gd_{u_1u_2}(u_2+u_3)+2\gd_{u_1u_3}(u_2+u_3)\\
&\iff\\
&2u_2(u_2u_1) +2u_3(u_3u_1) +2(u_1u_2)u_3+2(u_1u_3)u_2\\
&=-u_2^2u_1 -u_3^2u_1-2u_1(u_2u_3)+\gd_{u_2^2}u_1+\gd_{u_3^2}u_1+2\gd_{u_2u_3}u_1\\
&+2\gd_{u_1u_2}u_2+2\gd_{u_1u_2}u_3+2\gd_{u_1u_3}u_3+2\gd_{u_1u_3}u_2\\
&\overset{(i)}{\iff}\\
&2(u_1u_2)u_3+2(u_1u_3)u_2=-2u_1(u_2u_3) +2\gd_{u_2u_3}u_1+2\gd_{u_1u_3}u_2+2\gd_{u_1u_2}u_3\\
&\iff\\
&u_2(u_3u_1)+u_3(u_2u_1)+(u_2u_3)u_1=\gd_{u_2u_3} u_1+\gd_{u_1u_3}u_2+\gd_{u_1u_2}u_3.
\end{align*}
Where $(i)$ is obtained by applying Lemma \ref{lem u_1(u_1u_2)}(5) twice.
\medskip

\noindent
(2):\quad
This is obtained by replacing $u_2$ with $u_1$ and $u_3$ with $u_2$ in (1)
\medskip

\noindent
(3):\quad
Put  $u_2z$ in place of $u_2$ in (2).  
\medskip

\noindent
(4):\quad
Multiplying (2) by $z$ we have
\begin{align*}
&(u_1^2u_2)z=-2(u_1(u_1u_2))z+\gd_{u_1^2}u_2z+2\gd_{u_1u_2}u_1z\\
&=-2(u_1(\gd_{u_1u_2}e+z_{u_1u_2}))z+\gd_{u_1^2}u_2z+2\gd_{u_1u_2}u_1z\\
&=-\gd_{u_1u_2}u_1z-2(u_1z_{u_1u_2})z+\gd_{u_1^2}u_2z+2\gd_{u_1u_2}u_1z\\
&\overset{(i)}{=}-2u_1(z_{u_1u_2}z)+2(u_1z)z_{u_1u_2}+\gd_{u_1^2}u_2z+\gd_{u_1u_2}u_1z\\
&=-2u_1((u_1u_2)z)+2(u_1z)(u_1u_2-\gd_{u_1u_2}e)+\gd_{u_1^2}u_2z+\gd_{u_1u_2}u_1z\\
&=-2u_1((u_1u_2)z)+2(u_1z)(u_1u_2)+\gd_{u_1^2}u_2z.
\end{align*}
Where $(i)$ is obtained by using Lemma \ref{lem u(z_1z_2)}(1).
\medskip

\noindent
(5):\quad
Put $u_1$ in place of $u_2$ in (3).  We get
\begin{align*}
&u_1^2(u_1z)+2u_1(u_1(u_1z))=\gd_{u_1^2}u_1z+2\gd_{u_1(u_1z)}u_1\\
&\overset{(i)}{\iff} u_1^2(u_1z)+2u_1(\gd_{u_1(u_1z)}e+\half u_1^2z)=\gd_{u_1^2}u_1z+2\gd_{u_1(u_1z)}u_1\\
&\iff u_1^2(u_1z)+\gd_{u_1(u_1z)}u_1+u_1(u_1^2z)=\gd_{u_1^2}u_1z+2\gd_{u_1(u_1z)}u_1\\
&\overset{(ii)}{\iff} 2u_1^2(u_1z)=\gd_{u_1^2}u_1z+\gd_{u_1(u_1z)}u_1.
\end{align*}
Here $(i)$ holds by Theorem \ref{thm identities}(6), and $(ii)$ is by Lemma \ref{lem jordan u z}
below. (Indeed Lemma \ref{lem jordan u z} naturally belongs in \S 6.)
\end{proof}

%%%%%%%%%%%%%%%%%%%%%%%%%%%%%%%%%%%%%%%
%4.5
\begin{lemma}\label{lem (uz)^2}
%%%%%%%%%%%%%%%%%%%%%%%%%%%%%%
For $u, u_1, u_2\in U,\  z, z_1, z_2\in Z$ we have  
\begin{enumerate}
%1
\item
$2(uz)^2+u^2z^2=u(uz^2)+(u^2z)z.$

%2
\item
$4(u_1z)(u_2z)+2(u_1u_2)z^2=u_1(u_2z^2)+u_2(u_1z^2)+2((u_1u_2)z)z.$

%3
\item
$4(uz_1)(uz_2)+2u^2(z_1z_2)=2u(u(z_1z_2))+(u^2z_1)z_2+(u^2z_2)z_1.$

%4
\item
$2(uz)^2z+\half u^2z^2=((u^2z)z)z.$

%5
\item
$\gd_{u(uz^2)}=2\gd_{(uz)^2}.$

%6
\item
$4\gd_{(u_1z)(u_2z)}=\gd_{u_1(u_2z^2)+u_2(u_1z^2)}.$

%7
\item
$2\gd_{(uz_1)(uz_2)}=\gd_{u(u(z_1z_2))}.$
\end{enumerate}
\end{lemma}
\begin{proof}
Part (1) follows from Theorem \ref{thm identities}(4) using Theorem \ref{thm identities}(2\&6).  
For part (2)
replace $u$ by $u_1+u_2$ is (1), and then use (1) twice, to get:
\begin{align*}
&2((u_1+u_2)z)^2+(u_1+u_2)^2z^2=(u_1+u_2)((u_1+u_2)z^2)+((u_1+u_2)^2z)z\\
&\iff 2(u_1z)^2+2(u_2z)^2+4(u_1z)(u_2z)+u_1^2z^2+u_2^2z^2+2(u_1u_2)z^2\\
&=u_1(u_1z^2)+u_2(u_2z^2)+u_1(u_2z^2)+u_2(u_1z^2)+(u_1^2z)z+ (u_2^2z)z+2((u_1u_2)z)z\\
&\iff 4(u_1z)(u_2z)+2(u_1u_2)z^2=u_1(u_2z^2)+u_2(u_1z^2)+2((u_1u_2)z)z.
\end{align*}
This shows (2).
\medskip

\noindent
(3):\quad
Replacing $z$ with $z_1+z_2$ in (1) we get
\[
2(u(z_1+z_2))^2+u^2(z_1+z_2)^2=u(u(z_1+z_2)^2)+(u^2(z_1+z_2))(z_1+z_2).
\]
Or
\begin{align*}
&2(uz_1)^2+2(uz_2)^2+4(uz_1)(uz_2)+u^2z_1^2+u^2z_2^2+2u^2(z_1z_2)\\
&=u(uz_1^2)+u(uz_2^2)+2u(u(z_1z_2))+(u^2z_1)z_1+(u^2z_2)z_2+(u^2z_1)z_2+(u^2z_2)z_1.
\end{align*}
So using (1) we get (3).
\medskip

\noindent
(4):\quad
Multiply (1) by $z$ and note that $(u(uz^2))z=\half u^2z^2,$ by Theorem
\ref{thm identities}(6).

Parts (5), (6) and (7) are consequences of the previous parts.
\end{proof}

%%%%%%%%%%%%%%%%%%%%%%%%%%%%%%%%%%%%%%%%%%%%%%%%%%%%
%%%%%%%%%%%%%%%%%%%%%%%%%%%%%%%%%%%%%%%%%%%%%%%%
%%%%%%%%%%%%%%%%%%%%%%%%%%%%%%%%%%%%%%%%%%%%%%%%
%section5
\section{Half-axes in $A$}
%%%%%%%%%%%%%%%%%%%%%%%%%%%%%%%%%%%%%%%%%
%%%%%%%%%%%%%%%%%%%%%%%%%%%%%%%%%%%%%%%%
%%%%%%%%%%%%%%%%%%%%%%%%%%%%%%%%%%%%%%

Let $e\ne f\in A$ be another half-axis in $A$ and write
\[
f=\gc e+u_1+z\qquad\gc\in\ff,\ u_1\in U,\ z\in Z.
\]
Thus the subalgebra $A_{e,f}$ of $A$ generated by $e$ and $f$
is a primitive axial algebra of Jordan type half.
We use \cite{hss1} to deduce information on $A_{e,f}$.
We use the notation of \cite{hss1}.
Let
\[
\gs:=\gs_{e,f},\ \pi:=\pi_{e,f},\ \gd_1:=\gd_{u_1^2}\text{ and }z_1:=z_{u_1^2}.
\]

%%%%%%%%%%%%%%%%%%%%%%%%%%%%%%%%%%%%%
%5.1
\begin{lemma}\label{lem f^2=f}
%%%%%%%%%%%%%%%%%%%%%%%%%%%%%%%%
We have
\begin{enumerate}
\item
$u_1z=-\half(\gc-1)u_1;$

\item
$z_1=z-z^2;$

\item
$\gc-\gc^2=\gd_1.$
\end{enumerate}
\end{lemma}
\begin{proof}
We have
\begin{gather*}
f^2=\gc^2 e+u_1^2+z^2+\gc u_1+2u_1z=(\gc^2+\gd_1)e+(\gc u_1+2u_1z)+(z_1+z^2).
\end{gather*}
Since $f^2=f=\gc e+u_1+z,$ the lemma follows. 
\end{proof}

%%%%%%%%%%%%%%%%%%%%%%%%%%%%%%%%%%%%%%%%%%%%%
%5.2
\begin{prop}\label{prop f^2=f}
%%%%%%%%%%%%%%%%%%%%%%%%%%%%%%%%%%%%%%%%%%%
We have
\begin{enumerate}
%1
\item 
$\gs=\frac{\gc -1}{2}e-\half z$

%2
\item
$\pi=\frac{\gc-1}{2}.$

%3
\item
$z^2=(1-\gc)z.$

%4
\item
$z_1=\gc z,$ in particular

%5
\item
if $\gc\ne 0,$ then $u_1\ne 0,$ and $z=\frac{u_1^2}{\gc}-\frac{\gd_1}{\gc}e.$
 
%6
\item
If $\gc=0,$  then $f=u_1+z,$ with $u_1^2=0$ and $zu_1=\half u_1$.
In particular $u_1\in A_{1/2}(f)$. 
\end{enumerate}
\end{prop}
\begin{proof}
(1):\quad
We have $ef=\gc e+\half u_1,$ hence
\begin{gather*}
\gs=ef-\half e-\half f=\gc e+\half u_1-\half e-\half\gc e-\half u_1-\half z\\
=\half (\gc -1)e-\half z.
\end{gather*}
\medskip

\noindent
(2):\quad
$\gs e=\pi e\implies \half(\gc-1)e=\pi e,$ so (2) holds.
\medskip

\noindent
(3):\quad
Note that by (2), $\gs=\pi e-\half z,$ hence
\[
\gs^2=\pi\gs\implies \pi^2 e+\quar z^2=\pi^2 e-\half\pi z.
\]
Hence $\quar z^2=-\half\pi z$.  It follows that $z^2=-2\pi z=(1-\gc)z.$
\medskip

\noindent
(4):\quad
By Lemma \ref{lem f^2=f}(2) and by (3), we get $z_1=z-z^2=z+(\gc-1)z=\gc z.$
\medskip

\noindent
(5):\quad
If $u_1=0,$ then by Lemma \ref{lem f^2=f}(2\&3), $\gc=1$ and $z^2=z$.
But then $f$ is not a half-axis, since $fe=e$.  Now (5) follows from
(4) since $z_1=u_1^2-\gd_1e.$
\medskip

\noindent
(6):\quad
Suppose $\gc=0$.  Then, by Lemma \ref{lem f^2=f}(3), $\gd_1=0$. Also by (4), 
$z_1=0,$ so $u_1^2=0$.  Also, by \ref{lem f^2=f}(1),
$u_1z=\half u_1$.  Thus $fu_1=zu_1=\half u_1,$ and $u_1\in A_{1/2}(f).$
\end{proof}

%%%%%%%%%%%%%%%%%%%%%%%%%%%%%%%%%%%%%%%%
%%%%%%%%%%%%%%%%%%%%%%%%%%%%%%%%%%%%%%%%%%%%%%
%%%%%%%%%%%%%%%%%%%%%%%%%%%%%%%%%%%%%%%%%%%%%%%%
%%%%%%%%%%%%%%%%%%%%%%%%%%%%%%%%%%%%%%%%%%%%%%%%%%
%section6
\section{Special cases of the Jordan identity}
%%%%%%%%%%%%%%%%%%%%%%%%%%%%%%%%%%%%%%%%%%%%%%%
%%%%%%%%%%%%%%%%%%%%%%%%%%%%%%%%%%%%%%%%%%%%%%%
%%%%%%%%%%%%%%%%%%%%%%%%%%%%%%%%%%%%%%%%%%%%%%%

Throughout this section
we assume that the identities $x^2x^2=xx^3$ and  $x^3x^2=xx^4$ hold strictly in $A$.
In this section we deduce certain identities that
are specific cases of the general Jordan identity.
These indicate that $A$ ``tends'' to be a Jordan algebra.

%%%%%%%%%%%%%%%%%%%%%%%%%%%%%%%%%%%%%%%%%%%
%6.1
\begin{lemma}\label{lem jordan u z}
%%%%%%%%%%%%%%%%%%%%%%%%%%%%%%%%%%%%%%%%%
Let $u_1\in U$ and $z\in Z,$ then $u_1^2(u_1z)=u_1(u_1^2z).$
\end{lemma}
\begin{proof}
By Lemma \ref{lem u z1} we have
\[
u_1^2(u_1z)=(\gd_{u_1^2}e+z_{u_1^2})u_1z=\half \gd_{u_1^2} u_1z+(u_1z)z_{u_1^2}=(u_1z_{u_1^2})z+(u_1z)z_{u_1^2}.
\]
But also $u_1(u_1^2z)=u_1(z_{u_1^2}z)$.  So the lemma follows from
Lemma \ref{lem u(z_1z_2)}(1).
\end{proof}

%%%%%%%%%%%%%%%%%%%%%%%%%%%%%%%%%%%%%
%6.2
\begin{lemma}\label{lem u u}
%%%%%%%%%%%%%%%%%%%%%%%%%%%%%%%%%%%
Let $u_1, u_2\in U$ and write $u_1^2=\gd_1 e+z_1.$  
Assume that $u_1^2\notin Z$.  Then
\[
\gd_{u_1(u_2z_1)}=\half \gd_1\gd_{u_1u_2}.
\]
\end{lemma}
\begin{proof}
Recall from Lemma \ref{lem u z1}
that $z_1^2=\gd_1 z_1$ and $u_1z_1=\half \gd_1u_1$.
Hence also $u_1z_1^2=\half \gd_1^2 u_1$.
Thus, by Lemma \ref{lem (uz)^2}(6), with $z=z_1$ we have
\begin{align*}
&4\gd_{(u_1z_1)(u_2z_1)}=\gd_{u_1(u_2z_1^2)}+\gd_{u_2(u_1z_1^2)}\\
&\iff 2\gd_1\gd_{u_1(u_2z_1)}=\gd_1\gd_{u_1(u_2z_1)}+\half\gd_1^2\gd_{u_1u_2}\\
&\iff 2\gd_{u_1(u_2z_1)}=\gd_{u_1(u_2z_1)}+\half\gd_1\gd_{u_1u_2}.\qedhere
\end{align*}
\end{proof}

%%%%%%%%%%%%%%%%%%%%%%%%%%%%%%%%%%%%%%
%6.3
\begin{lemma}\label{lem jordan u u}
%%%%%%%%%%%%%%%%%%%%%%%%%%%%%%%%%%%%
Let $u_1\in U$ such that $u_1^2\notin Z$.  Then $u_1^2(u_1u_2)=u_1(u_1^2u_2).$
\end{lemma}
\begin{proof}
Write $u_1^2=\gd_1e+z_1$. We have
\[
u_1^2(u_1u_2)=(\gd_1e+z_1)(u_1u_2)=\gd_1e(u_1u_2)+(u_1u_2)z_1,
\]
and using Lemma \ref{lem u z1},
\begin{gather*}
u_1(u_1^2u_2)=u_1((\gd_1e+z_1)u_2)=\half \gd_1(u_1u_2)+u_1(u_2z_1)\\
=u_2(u_1z_1)+u_1(u_2z_1).
\end{gather*}
Hence we need to show that
\[
\gd_1e(u_1u_2)+(u_1u_2)z_1=u_2(u_1z_1)+u_1(u_2z_1).
\]
By Lemma \ref{lem (u_1u_2)z},
\[
(u_1u_2)z_1=u_1(u_2z_1)+u_2(u_1z_1)-e(u_1(u_2z_1)+u_2(u_1z_1)),
\]
so we must show that
\[
\gd_1e(u_1u_2)=e(u_2(u_1z_1)+u_1(u_2z_1))=\half\gd_1 e(u_1u_2)+e(u_1(u_2z_1)).
\]
Or
\[
\half\gd_1e(u_1u_2)=e(u_1(u_2z_1)).
\]
But this was shown in Lemma \ref{lem u u}.
\end{proof}

%%%%%%%%%%%%%%%%%%%%%%%%%%%%%%%%%%%%%%%%%%%%%
%6.4
\begin{lemma}\label{lem jordan z z}
%%%%%%%%%%%%%%%%%%%%%%%%%%%%%%%%%%%%%%%%%
Let $u\in U$ and $z\in Z,$ then
$(u^2z^2)z=(u^2z)z^2$.
\end{lemma}
\begin{proof}
By Lemma \ref{lem (uz)^2}(4),  we have
%eeeeeeeeeeeeeeeeeeeeeeeeeeeeeeeeeeeeeeeeeeeeeeeeeeee
%eq6.1
\begin{equation}\label{eq lem Z 1}
2(uz)^2z+\half(u^2z^2)z=((u^2z)z)z.
\end{equation}
%eeeeeeeeeeeeeeeeeeeeeeeeeeeeeeeeeeeeeeeeeeeeeeee

Next, by Lemma \ref{lem (uz)^2}(2), putting $u$ in place of $u_1$ and $uz$
in place of $u_2$ we get.
\[
4(uz)((uz)z)+2(u(uz))z^2=u((uz)z^2)+(uz)(uz^2)+2((u(uz))z)z.
\]
Using Theorem \ref{thm identities}(2\&6) and Lemma \ref{lem u(z_1z_2)}(2)
we get
%eeeeeeeeeeeeeeeeeeeeeeeeeeeeeeeeeeeeeeeeeeeeeeeeeeeee
%eq6.2
\begin{equation}\label{eq lem Z 2}
(uz)(uz^2)+(u^2z)z^2=\half u(uz^3)+((u^2z)z)z.
\end{equation}
%eeeeeeeeeeeeeeeeeeeeeeeeeeeeeeeeeeeeeeeeeeeeeeeee
Note that the $Z$-part of $(uz)(uz^2)=2(uz)((uz)z)$ is $(uz)^2z,$ using Theorem
\ref{thm identities}(2\&6).  Comparing the $Z$-parts in equation \eqref{eq lem Z 2}
(using again \ref{thm identities}(6)) we see that
%eeeeeeeeeeeeeeeeeeeeeeeeeeeeeeeeeeeeeeeeeeeeee
%eq6.3
\begin{equation}\label{eq lem Z 3}
(uz)^2z+(u^2z)z^2=\quar u^2z^3+((u^2z)z)z.
\end{equation}
%eeeeeeeeeeeeeeeeeeeeeeeeeeeeeeeeeeeeee
Note now that by Lemma \ref{lem (u_1u_2)z} and Lemma \ref{lem u(z_1z_2)}(2),  $(u(uz))z^2$ equals the $Z$-part of
\[
(uz)(uz^2)+u((uz)z^2)=2(uz)((uz)z)+u((uz)z^2)=2(uz)((uz)z)+\half u(uz^3).
\]
Using Theorem \ref{thm identities}(6)   we get that
\[
\half(u^2z)z^2=(uz)^2z+\quar u^2z^3.
\]
or
%eeeeeeeeeeeeeeeeeeeeeeeeeeeeeeeeeeeeeeeeeeeeeeeeeeeeeeee
%eq6.4
\begin{equation}\label{eq lem Z 4}
\quar u^2z^3=\half(u^2z)z^2-(uz)^2z.
\end{equation} 
%eeeeeeeeeeeeeeeeeeeeeeeeeeeeeeeeeeeeeeeee
%
Inserting equation \eqref{eq lem Z 4} in equation \eqref{eq lem Z 3} we get
%
%eeeeeeeeeeeeeeeeeeeeeeeeeeeeeeeeeeeeeeeeeeeeeeee
%eq6.5
\begin{equation}\label{eq lem Z 5} 
2(uz)^2z+\half(u^2z)z^2=((u^2z)z)z.
\end{equation}
%eeeeeeeeeeeeeeeeeeeeeeeeeeeeeeeeeeeeeeeeeeee
%
Comparing equations \eqref{eq lem Z 1} and \eqref{eq lem Z 5} we get the lemma.
\end{proof}

%%%%%%%%%%%%%%%%%%%%%%%%%%%%%%%%%%%%%%%%%%%%%%%%%%%%
%%%%%%%%%%%%%%%%%%%%%%%%%%%%%%%%%%%%%%%%%%%%%%%%%%
%%%%%%%%%%%%%%%%%%%%%%%%%%%%%%%%%%%%%%%%%%%%%%%%%
%section7
\section{Identities between eigenspaces of ${\rm ad}_e,$ and The Jordan identity}
%%%%%%%%%%%%%%%%%%%%%%%%%%%%%%%%%%%%%%%%%%%%%%
%%%%%%%%%%%%%%%%%%%%%%%%%%%%%%%%%%%%%%%%%%%%%
%%%%%%%%%%%%%%%%%%%%%%%%%%%%%%%%%%%%%%%%%%% 

The purpose of this section is to prove the following Theorem.

%%%%%%%%%%%%%%%%%%%%%%%%%%%%%%%%%%%%%%%%%%%%%%
%7.1
\begin{thm}\label{thm jordan}
%%%%%%%%%%%%%%%%%%%%%%%%%%%%%%%%%%%%%%%%
Let $A$ be a commutative non-associative algebra over $\ff$ and suppose
$e \in A$ is a half-axis of $A$. Then $A$ is a Jordan algebra if and only if the following identities
hold, for all elements $u, v \in A_{1/2}(e), z, z' \in A_0(e).$
\begin{enumerate}
\item
$u(u^2e)=\half u^3.$
\item
$u(zz')=(uz)z'+(uz')z.$ 
\item
$u(u^2z)=u^2(uz).$ 
\item
$u^2(zz')+2(uz)(uz') =(u^2z')z+2u((uz)z').$
\item
$(uz')z^2+2(uz)(zz') =u(z^2z')+2((uz)z')z.$
\item
$z(z^2z')=z^2(zz').$
\item
$2((uz)v)e+(uv)z=u(vz)+v(uz).$
\item
$u(u^2v)= u^2(uv).$ 
\item
$2u(v(uz))+(u^2v)z=2(uv)(uz)+u^2(vz).$
\item
$2((uz)v)z+u(vz^2)=2(uz)(vz)+(uv)z^2.$
\item
$z(z^2v)=z^2(zv).$
\end{enumerate}
\end{thm}
\begin{proof}
As noted in the introduction, since tensor products commute with direct sums, $e$ is a half-axis in
$A$ if and only if it is a half axis in $A_{\kk},$ for any field extension $\kk/\ff$.

Let $r, s, t$ be independent variables and $\kk := \ff(r, s, t)$ the corresponding rational
function field. We canonically identify $A \subseteq A_\kk$ as an $\ff$-subalgebra and, for arbitrary
elements $u \in U := A_{1/2}(e), z \in Z := A_0(e),$ consider the quantity
\[
x := re + su + tz \in A_\kk.
\]
Note that the fusion rules imply
\[
x^2 = r^2e + rsu+s^2u^2+2stuz+t^2z^2.
\]
Since Jordan algebras are stable under base change, A is a Jordan algebra if and only if $A_\kk$
is a Jordan algebra.  Hence $A$ is a Jordan algebra if and only if
for all $u, v \in U, z, z' \in Z$ and $x$ as above, the relations
\begin{align}\label{eq x(x^2e)}
x(x^2e) &= x^2(xe).\\\label{eq x(x^2z)}
x(x^2z') &= x^2(xz').\\\label{eq x(x^2u)}
x(x^2v) &= x^2(xv).
\end{align}
hold. Expanding both sides of equations \eqref{eq x(x^2e)}--\eqref{eq x(x^2u)}, viewing the result as polynomials in $r, s, t,$
over $A$, and comparing coefficients, straightforward computations, which are only mildly
tedious and are left to the reader, show that equations \eqref{eq x(x^2e)}--\eqref{eq x(x^2u)} are equivalent to
(1)--(11) of the theorem. 
\end{proof}

As a corollary we get the following theorem

%%%%%%%%%%%%%%%%%%%%%%%%%%%%%%%%%
%7.2
\begin{thm}\label{thm jordan 2}
%%%%%%%%%%%%%%%%%%%%%%%%%%%%%%%
Let $A$ be a commutative non-associative algebra over $\ff$, $e \in A$ a half-axis, 
and suppose the identities $x^2x^2 = xx^3, x^3x^2 = xx^4$ hold strictly in $A$.
Then A is a Jordan algebra if and only if the following identities hold
for all $u_1,u_2\in U$ and $z_1,z_2\in Z:$
\begin{itemize}
%1
\item[$(i)$]
$u_1(u_1^2u_2)=u_1^2(u_1u_2).$

%2
\item[$(ii)$]
$(u_1^2u_2)z_1+2u_1(u_2(u_1z_1))=u_1^2(u_2z_1)+(u_1u_2)(u_1z_1).$

%3
\item[$(iii)$]
$(u_1^2z_2)z_1+2u_1((u_1z_1)z_2)=u_1^2(z_1z_2)+2(u_1z_1)(u_1z_2).$

%4
\item[$(iv)$]
$u_1(u_2z_1^2)+2(u_2(u_1z_1))z_1=(u_1u_2)z_1^2+2(u_1z_1)(u_2z_1).$

%5
\item[$(v)$]
$z_1(z_1^2z_2)=z_1^2(z_1z_2).$ 

%6
\item[$(vi)$]
$(u_1u_2)z_1+2 e(u_2(u_1z_1))=u_1(u_2z_1)+u_2(u_1z_1).$
\end{itemize}
\end{thm}
\begin{proof}
After a change of notation, the identities $(i)$--$(vi)$  agree with identities
(8), (9), (4), (10), (6), (7) of Theorem \ref{thm jordan}. Hence Theorem \ref{thm jordan} shows that, 
if $A$ is a Jordan algebra, then $(i)$--$(vi)$ hold.

Conversely, suppose $(i)$--$(vi)$ hold. We must prove identities (1), (2),
(3), (5) and (11) of Theorem \ref{thm jordan}.
We use Theorem \ref{thm identities}.

Identity (1) of Theorem \ref{thm jordan} follows from Theorem \ref{thm identities}(1) and Lemma \ref{lem u z1},
while identity (2) is Lemma \ref{lem u(z_1z_2)}(1), and
identity (3) is  Lemma \ref{lem jordan u z}. Also identity (11) is Lemma \ref{lem u(z_1z_2)}(2).
 
Finally we prove identity (5) of Theorem \ref{thm jordan}.
We claim that
%eeeeeeeeeeeeeeeeeeeeeeeeeeeeeeeeeeeeeeeeeee
\begin{equation}\label{eq*}
2(uz)(zz')=(uz^2)z'+2((uz)z')z.
\end{equation}
%eeeeeeeeeeeeeeeeeeeeeeeeeeeeeeeeeeeeeeeeeeeee
Indeed using Theorem \ref{thm jordan}(2) twice we get
\[
2(uz)(zz')=2((uz)z)z'+2((uz)z')z=(uz^2)z'+2((uz)z')z.
\]
Combining identities \eqref{eq*} with Theorem \ref{thm jordan}(2), we obtain
\[
(uz')z^2 + 2(uz)(zz') = (uz')z^2 +(uz^2)z'+2((uz)z')z=u(z^2z')+2((uz)z')z,
\]
as desired. 
\end{proof}

%%%%%%%%%%%%%%%%%%%%%%%%%%%%%%%%%%%%%%%%%%%%%%%%
%%%%%%%%%%%%%%%%%%%%%%%%%%%%%%%%%%%%%%%%%%%%%%%%%%
%%%%%%%%%%%%%%%%%%%%%%%%%%%%%%%%%%%%%%%%%%%%%%%%%%%%%
%section8
\section{A condition for $A$ to be a Jordan algebra}\label{sect *}
%%%%%%%%%%%%%%%%%%%%%%%%%%%%%%%%%%%%%%%%%%%%%%%%%
%%%%%%%%%%%%%%%%%%%%%%%%%%%%%%%%%%%%%%%%%%%%%%%
%%%%%%%%%%%%%%%%%%%%%%%%%%%%%%%%%%%%%%%%%%%%%

Throughout this section
we assume that the identities $x^2x^2=xx^3$ and  $x^3x^2=xx^4$ hold strictly in $A$.
We consider the condition
\[\tag{$*$}
\gvp_e(u_1(u_2z))=\gvp_e(u_2(u_1z)),\text{ for all }u_1,u_2\in U\text{ and }z\in Z.
\]

In this section we will  prove:

%%%%%%%%%%%%%%%%%%%%%%%%%%%%%%%%%%%%%%%%%%%%
%8.1
\begin{thm}\label{thm main s8}
%%%%%%%%%%%%%%%%%%%%%%%%%%%%%%%%%%%%%%%%%
Assume that $(*)$ holds.  Then identities (i)--(iv)
and identity (vi) of Theorem \ref{thm jordan 2}
hold, for all $u_1,u_2\in U$ and $z_1, z_2\in Z.$
\end{thm}

Using Theorem \ref{thm main s8} we easily deduce:

%%%%%%%%%%%%%%%%%%%%%%%%%%%%%%%%%%%%%%%
%8.2
\begin{thm}\label{thm jordan u}
%%%%%%%%%%%%%%%%%%%%%%%%%%%%%%%%%%%%%%
Assume that $(*)$ holds.  Then $x(x^2u)=x^2(xu),$ for all $x\in A$ and $u\in U.$
\end{thm}
 
%%%%%%%%%%%%%%%%%%%%%%%%%%%%%%%%%%%%%%%%%%%%%%
%8.3
\begin{thm}\label{thm jordan z}
%%%%%%%%%%%%%%%%%%%%%%%%%%%%%%%%%%%%%%%
Assume that $(*)$ holds and let $z\in Z$.  Then the following are equivalent
\begin{itemize}
\item[(i)] 
$x(x^2z)=x^2(xz),$ for all $x\in A.$

\item[(ii)]
$z_1(z_1^2z)=z_1^2(z_1z),$ for all $z_1\in Z.$
\end{itemize}
\end{thm}

%%%%%%%%%%%%%%%%%%%%%%%%%%%%%%%%%%%%%
%8.4
\begin{thm}\label{thm jordan if Z}
%%%%%%%%%%%%%%%%%%%%%%%%%%%%%%%%%%%%
Assume that $(*)$ above holds.
Then $A$ is a Jordan algebra if and only if $Z$ is a Jordan algebra.
\end{thm}
 
\noindent
In the remainder of this section we assume that $(*)$ holds.

%%%%%%%%%%%%%%%%%%%%%%%%%%%%%%%%%%%%%%%%%%%%%%%%%%%
%8.5
\begin{remark}
%%%%%%%%%%%%%%%%%%%%%%%%%%%%%%%%%%%%%%%%%%%%%%%%%%
One of our motivation for hypothesis $(*)$ above is that it holds
in any primitive axial algebra of Jordan type half (not necessarily power associative).
Indeed, suppose that $A$ is such an algebra.
By \cite[Theorem 4.1]{hss2} $A$ admits a Frobenius form
$(\cdot\, ,\, \cdot)$ such that $(f, x)=\gvp_f(x),$
for all half-axes $f\in A$.  Hence
\[
\gvp_e(u_1(u_2z))=(e,u_1(u_2z))=(eu_1,u_2z)=\half (u_1,u_2z)=\half(u_1u_2,z).
\]
By symmetry, $\gvp_e(u_2(u_1z))=\half(u_1u_2,z),$ so indeed $(*)$ holds in $A.$
\end{remark}

%%%%%%%%%%%%%%%%%%%%%%%%%%%%%%%%%%%%%%%%
%8.6
\begin{prop}\label{prop 1st identity}
%%%%%%%%%%%%%%%%%%%%%%%%%%%%%%%%%%%%%%%
Identity (i) of Theorem \ref{thm jordan 2} holds, namely
\[
u_1(u_1^2u_2)=u_1^2(u_1u_2),
\]
for all $u_1, u_2\in U$.
\end{prop}
\begin{proof}
If $u_1^2\notin Z,$ this is Lemma \ref{lem jordan u u}.  So suppose $u_1^2\in Z$.
Then $\gd_{u_1^2}=0,$ so, by Theorem \ref{thm identities}(1), $u_1^3=0$.
By $(*)$ we have $\gvp_e(u_1(u_2u_1^2))=\gvp_e(u_2(u_1u_1^2))=0$.  Thus $u_1(u_2u_1^2)\in Z$.
By Lemma \ref{lem (u_1u_2)z}, 
\[
(u_1u_2)u_1^2=u_1(u_2u_1^2)+u_2(u_1u_1^2)=u_1(u_2u_1^2).\qedhere
\]
\end{proof}

%%%%%%%%%%%%%%%%%%%%%%%%%%%%%%%%%%%%%%%%%%%%%
%8.7
\begin{prop}\label{prop 2nd identity}
%%%%%%%%%%%%%%%%%%%%%%%%%%%%%%%%%%%%%%%%%%%%%
Identity (ii) of Theorem \ref{thm jordan 2} holds, namely
\[
(u_1^2u_2)z+2(u_1(u_2(u_1z))= u_1^2(u_2z)+2 (u_1u_2)(u_1z),
\]
for all $u_1, u_2\in U$ and $z\in Z$.
\end{prop}
\begin{proof}
Let $\gd_1:=\gd_{u_1^2}$. We have
\begin{align*}
&(u_1^2u_2)z+2u_1(u_2(u_1z))\\
&\overset{(i)}{=}-2u_1((u_1u_2)z)+2(u_1u_2)(u_1z)+2u_1(u_2(u_1z))+\gd_1u_2z\\
&=2u_1(-(u_1u_2)z+u_2(u_1z))+2(u_1u_2)(u_1z)+\gd_1u_2z\\
&\overset{(ii)}{=}2u_1(-u_1(u_2z)+2\gd_{u_2(u_1z)}e)+2(u_1u_2)(u_1z)+\gd_1u_2z\\
&=-2u_1(u_1(u_2z))+2(u_1u_2)(u_1z)+2\gd_{u_2(u_1z)}u_1+\gd_1u_2z\\
&\overset{(iii)}{=}u_1^2(u_2z)-\gd_1u_2z-2\gd_{u_2(u_1z)}u_1+2(u_1u_2)(u_1z)+2\gd_{u_2(u_1z)}u_1+\gd_1u_2z\\
&=u_1^2(u_2z)+2(u_1u_2)(u_1z).
\end{align*}
Where equality (i) comes from Lemma \ref{lem u_1u_2u_3}(4), equality (ii) comes from
Lemma \ref{lem (u_1u_2)z} and $(*)$.  Finally, equality (iii)
comes from Lemma \ref{lem u_1u_2u_3}(3).
\end{proof}

%%%%%%%%%%%%%%%%%%%%%%%%%%%%%%%%
%8.8
\begin{prop}\label{prop 3rd identity}
%%%%%%%%%%%%%%%%%%%%%%%%%%%%%%%%%
Identity (iii) of Theorem \ref{thm jordan 2} holds, namely
\[
(u_1^2z_2)z_1+2u_1((u_1z_1)z_2)=u_1^2(z_1z_2)+2(u_1z_1)(u_1z_2),
\]
for all $u_1\in U$ and $z_1, z_2\in Z$.
\end{prop}
\begin{proof}
We use Theorem \ref{thm identities}(6) and Lemma \ref{lem (u_1u_2)z}:
\begin{align*}
&(u_1^2z_2)z_1+2u_1((u_1z_1)z_2)=2(u_1(u_1z_2))z_1+2u_1((u_1z_1)z_2)\\ 
&=2(u_1z_1)(u_1z_2)+2u_1((u_1z_2)z_1)+2u_1((u_1z_1)z_2)-2\gd_{u_1((u_1z_2)z_1)}e-2\gd_{(u_1z_1)(u_1z_2)}e\\
&= 2(u_1z_1)(u_1z_2)+2u_1(u_1(z_1z_2))-2\gd_{u_1((u_1z_2)z_1)}e-2\gd_{(u_1z_1)(u_1z_2)}e\\
&=2(u_1z_1)(u_1z_2)+u_1^2(z_1z_2)+2\gd_{u_1(u_1(z_1z_2))}e-2\gd_{u_1((u_1z_2)z_1)}e-2\gd_{(u_1z_1)(u_1z_2)}e.\\
\end{align*}
Note now that by $(*),$ 
\[
\gd_{u_1((u_1z_2)z_1)}e=\gd_{(u_1z_1)(u_1z_2)}.
\]
Hence we get
\begin{gather*}
(u_1^2z_2)z_1+2u_1((u_1z_1)z_2)\\
=2(u_1z_1)(u_1z_2)+u_1^2(z_1z_2)+2\gd_{u_1(u_1(z_1z_2))}e-4\gd_{(u_1z_1)(u_1z_2)}e.
\end{gather*}
But by Lemma \ref{lem (uz)^2}(6), $2\gd_{u_1(u_1(z_1z_2))}=4\gd_{(u_1z_1)(u_1z_2)},$
so we are done.
\end{proof}

%%%%%%%%%%%%%%%%%%%%%%%%%%%%%%%%%%%%%%%%%%%%%%%%%%
%8.9
\begin{prop}\label{prop 4th identity}
%%%%%%%%%%%%%%%%%%%%%%%%%%%%%%%%%%%%%%%%%%%%
Identity (iv) of Theorem \ref{thm jordan 2} holds, namely
\[
u_1(u_2z^2)+2(u_2(u_1z))z=(u_1u_2)z^2+2(u_1z)(u_2z),
\]
for all $u_1, u_2\in U$ and $z\in Z$.
\end{prop}
\begin{proof}
We have
\begin{align*}
&u_1(u_2z^2)+2(u_2(u_1z))z\\
&\overset{(i)}{=}(u_1u_2)z^2-u_2(u_1z^2)+2\gd_{u_1(u_2z^2)}e+2(u_2(u_1z))z\\
&\overset{(ii)}{=}(u_1u_2)z^2-2u_2((u_1z)z)+2\gd_{u_1(u_2z^2)}e+2(u_2(u_1z))z\\
&\overset{(iii)}{=}(u_1u_2)z^2-2\Big((u_2(u_1z))z-(u_1z)(u_2z)+2\gd_{(u_1z)(u_2z)}e\Big)\\
&+2\gd_{u_1(u_2z^2)}e+2(u_2(u_1z))z\\
&=(u_1u_2)z^2+2(u_1z)(u_2z)-4\gd_{(u_1z)(u_2z)}e+2\gd_{u_1(u_2z^2)}e\\
&\overset{(iv)}{=}(u_1u_2)z^2+2(u_1z)(u_2z)-4\gd_{(u_1z)(u_2z)}e+4\gd_{u_1((u_2z)z)}e\\
&\overset{(v)}{=}(u_1u_2)z^2+2(u_1z)(u_2z).
\end{align*}
Where equalities $(i)$ and $(iii)$ come from Lemma \ref{lem (u_1u_2)z} (and $(*)$),
equalities $(ii)$ and $(iv)$ come from the fact that $2(uz)z=uz^2,$
and equality $(v)$ comes from $(*)$.
\end{proof}

%%%%%%%%%%%%%%%%%%%%%%%%%%%%%%%%%%%%%%%%%%%%%%
%8.10
\begin{lemma}\label{lem 6th identity}
%%%%%%%%%%%%%%%%%%%%%%%%%%%%%%%%%%%%%%%%%%%
Identity (vi) of Theorem \ref{thm jordan 2} holds, namely
\[
(u_1u_2)z+2e(u_2(u_1z))=u_1(u_2z)+u_2(u_1z),
\]
for all $u_1, u_2\in U$ and $z\in Z$.
\end{lemma}
\begin{proof}
This is immediate from $(*)$ and Lemma \ref{lem (u_1u_2)z}.
\end{proof}

%%%%%%%%%%%%%%%%%%%%%%%%%%%%%%%%%%%%%%%%%%
\begin{proof}[Proofs of Theorems \ref{thm main s8}, \ref{thm jordan u}, \ref{thm jordan z} and \ref{thm jordan if Z}]\hfill
%%%%%%%%%%%%%%%%%%%%%%%%%%%%%%%%%%%%%%
\medskip

\noindent
Theorem  \ref{thm main s8}  is immediate from Theorem \ref{thm jordan 2} and
Propositions \ref{prop 1st identity}, \ref{prop 2nd identity}, \ref{prop 3rd identity}, \ref{prop 4th identity},
and Lemma \ref{lem 6th identity}.  Furthermore these propositions and Lemma show that if we
take  $\gb=z_2=0$ in Theorem \ref{thm jordan 2}, then all identities in that proposition are satisfied,
for all $u_1, u_2\in U$ and $z_1\in Z,$
so   Theorem \ref{thm jordan 2}
completes the proof of Theorem \ref{thm jordan u} (taking $y=u_2=u$).  The proof of Theorem \ref{thm jordan z} is similar.
Finally Theorem \ref{thm jordan if Z} is immediate from Theorem \ref{thm main s8} and Theorem \ref{thm jordan 2}.
\end{proof}
\medskip

\noindent
{\bf Acknowledgement.}
The referee report of this paper can be considered as an article
on its own right.  It upgraded the level of this paper and improved
the proofs in many parts. For example Theorem \ref{thm jordan}
is due to the referee. The author thinks that there is only
one mathematician in the world (in areas related to this paper) 
that could produce such a fantastic report.
\noindent

%%%%%%%%%%%%%%%%%%%%%%%%%%%%%%%%%%%%%
%%%%%%%%%%%%%%%%%%%%%%%%%%%%%%%%%%%%%%%
%%%%%%%%%%%%%%%%%%%%%%%%%%%%%%%%%%%%%%%%

\end{document}